\documentclass{article}

\usepackage{titlesec, blindtext, color}
\usepackage{tikz-cd}

\usepackage{mathtools}
\usepackage{amsmath}
\usepackage{amsfonts}
\usepackage{amssymb}
\usepackage{amsthm}
\usepackage{graphicx}
\usepackage{dsfont}
\usepackage{hyperref}
\usepackage{mathabx}
\usepackage{tikz}
\usepackage{comment}
\usepackage{float}
\usepackage{faktor}

\usetikzlibrary{calc,decorations.pathmorphing,shapes}

\usepackage{url}
\makeatletter
\g@addto@macro{\UrlBreaks}{\UrlOrds}
\makeatother
\usepackage{multicol}
\usepackage{lineno}
\usepackage{makeidx}

\theoremstyle{definition}
\newtheorem{defi}{Definition}[section]
\newtheorem{thm}[defi]{Theorem}
\newtheorem{lem}[defi]{Lemma}
\newtheorem{cor}[defi]{Corollary}

\newtheorem{quest}{Question}
\newtheorem{claim}[defi]{Claim}

\newtheorem{fact}[defi]{Fact}

\title{Borel chromatic numbers of locally countable $F_\sigma$ graphs and forcing with superperfect trees}
\author{Raiean Banerjee and Michel Gaspar}
\date{}

\begin{document}

\maketitle

\begin{abstract}
    In this work we study the uncountable Borel chromatic numbers, defined in \cite{geschke2011} as cardinal characteristics of the continuum, of low complexity graphs. We show that locally countable graphs with compact totally disconnected set of vertices have Borel chromatic number bounded by the continuum of the ground model. From this, we answer a question from Geschke and the second author (\cite{gasgesch2022}), and another question from Fisher, Friedman and Khomskii (\cite{Fischer2014CichosDR}) concerning regularity properties of subsets of the real line.
\end{abstract}

\section*{Acknowledgements}

We would like to thank our supervirors Stefan Geschke and Benedikt Löwe for their invaluable comments on this work. 

\section{Introduction}

In this article, we shall study set theoretic independence results for combinatorial statements about locally countable $F_\sigma$ graphs and their consequences for the set theory of the real numbers. The systematic study of definable graphs started in \cite{kechris1999borel} as a descriptive set-theoretic approach to concepts and results from graph theory, and this field is nowadays called \emph{descriptive graph combinatorics}.

If $G$ is a graph on a set $X$, then an $\alpha$-coloring of $G$ is a function $c: X \rightarrow \alpha$ such that $c(x) \neq c(y)$, for all $(x, y) \in G$, and ordinal $\alpha \geq 1$. Now if $X$ is endowed with a Polish topology, then $c$ is called a \emph{Borel coloring} if,  additionally, $c^{-1}(\{\beta\})$ is a Borel set, for every $\beta < \alpha$ --- i.e., every \emph{maximally monochromatic} set is Borel. The Borel chromatic number of $G$, denoted by $\chi_B (G)$, is the least cardinality of an ordinal $\alpha$ for which there exists a Borel $\alpha$-coloring of $G$.

Since we are working only with graphs on Polish spaces, our Borel chromatic numbers are bounded by $2^{\aleph_0}$. We will later see that, when uncountable, Borel chromatic numbers may assume different values in different models of set theory.

At the heart of the field of descriptive graph combinatorics is the $G_0$-di-chotomy: it says that there exists a closed graph which is minimal for analytic graphs of uncountable Borel chromatic numbers, thus $\chi_B (G_0)$ is the smallest possible uncountable Borel chromatic number for an analytic graph. 

For ZFC results about Borel chromatic numbers of various classes of graphs, the reader is refered to the founding article \cite{kechris1999borel}. 

In \cite{geschke2011} these numbers were first studied from the lens of consistency results, and in \cite{gasgesch2022} various Borel chromatic numbers of graphs were computed in models of set theory obtained by forcing with uniform tree-forcing notions. Of particular interest is the $F_\sigma$ equivalence relation $E_0$ on $2^\omega$, defined by
\[x E_0 y \leftrightarrow \forall^\infty n\ (x(n) = y(n)).\]
We also may think of equivalence relations as graphs with the identity included in order to make sense of their Borel chromatic numbers. This relation has been extensively studied by Kechris, Harrington and Louveau (e.g., \cite{harrington1990glimm}), Khomskii and Brendle (e.g., \cite{brendle2012polarized}) etc. 

Geschke and the second author asked  \cite[Question 5.2]{gasgesch2022} whether $\chi_B(E_0)$ is consistently smaller than the bounding number $\mathfrak{b}$, another of the mentioned cardinal characteristics of the continuum. 

We devise a new stronger notion of local countability for a graph, \emph{$\ell$-unbounded\-ness}, and show that if $X$ is a compact totally disconnected Polish space, and $G$ is an $F_\sigma$ $\ell$-unbounded graph on $X$, then after iteratively adding any number of Laver reals to a model of set theory, there still exists a ground model Borel coloring of $G$ (see Theorem \ref{laverthm}, item b). If we replace ``$\ell$-unboundedness'' for ``local countability'', then we have the same result by adding Miller reals instead. Theorem \ref{laverthm} item a was independently proved by Zapletal, but for closed graphs instead. His methods rely on the heavy machinery of his \emph{idealized forcing} (see \cite{zapletal2008forcing}), as well as iterable properties for ``sufficiently definable and homogeneous ideals''. The approach we take here is completely different and we resort only to classical combinatorical arguments of the forcings involved.

This answers positively the aforementioned question: in the model obtained by adding $\aleph_2$ Laver reals to a model of CH, $\chi_B (E_0) < \mathfrak{b}$ (see Corollary \ref{chiE0b}). As another consequence of our result, we shall be able to solve an open problem in the field of set theory of the real numbers: in this field, researchers usually study \emph{regularity properties}, i.e., properties of good behaviour of sets of real numbers and whether different properties can be \emph{separated}. If $\Gamma_0$ and $\Gamma_1$ are classes of sets of reals and $P$ and $Q$ are two regularity properties, we say that a model \emph{separates $\Gamma_0(P)$ from $\Gamma_1(Q)$} if in
this model every set in $\Gamma_1$ has property $Q$, but there is a set in $\Gamma_0$ that does not have property $P$. Many of these separation results are known for classes on the second level of the projective hierarchy. 

Fischer, Friedman and Khomskii asked \cite[Question 6.3]{Fischer2014CichosDR} whether it is possible to separate the Silver measurability of all $\boldsymbol{\Delta}^1_2$ sets from the Laver measurability of all $\boldsymbol{\Sigma}^1_2$ sets. This question has also been mentioned open by Brendle and L\"{o}we in \cite[Fig. 1]{eventbl} and by Ikegami in \cite[Fig. 2.1]{Ikegami2010GamesIS}

Mostly through the the work of Ikegami \cite[Theorem 1.3]{ike1}, it is now known that these notions of measurability at the second level of the projective hierarchy depend on the amount of generic reals that are added to $L$, the constructible universe.

Our result, together with Ikegami's theorem, answers the question from Fischer, Friedman and Khomskii positively: in the model obtained by iteratively adding $\aleph_1$ Laver reals to $L$, the constructible universe, all $\boldsymbol{\Sigma}^1_2$ sets are Laver measurable, but not all $\boldsymbol{\Delta}^1_2$ sets are Silver measurable (see Corollary \ref{silvervslaver}).

Note that this also answers a question of Brendle, Halbeisen and L\"owe: they asked whether the existence of sufficiently many splitting reals over $L$ already implies the Silver measurability of all $\boldsymbol{\Delta}^1_2$ sets \cite[Question 2]{brendle_halbeisen_lwe_2005}. The answer is `No', from Corollary \ref{splitting}.

In the next section we discuss the important feature of \textit{minimality} possessed by Laver forcing, where the key notion of \emph{guiding real} is introduced. In Section \ref{mainsection} we prove the single step version of Theorem \ref{laverthm} (Lemma \ref{laverlem}) and iterations of superperfect forcings. We discuss the applications to regularity properties in Section \ref{regularities}, and mention a few open questions in Section \ref{questions}.

\section{Guiding reals and minimality}\label{minimality}

A tree $p \subseteq \omega^{<\omega}$ is a \textit{Laver tree} iff every $t \in p$ extending the stem has infinitely many immediate successors in $p$. The \textit{Laver forcing}, $\mathbb{L}$, consists of Laver trees ordered by inclusion. 

Let us fix some notation: for a Laver tree $p \subseteq \omega^{<\omega}$,  there exists a natural bijection $\sigma \mapsto \sigma^\ast$ from $\omega^{<\omega}$ to its set of splitting nodes, $\mathrm{spl}(p)$, such that $\emptyset^\ast$ is the stem of $p$; and for each $n \in \omega$, $(\sigma^\smallfrown n)^\ast$ is the minimal splitting node of $p$ extending the $n$-th element of $\mathrm{succ}_p (\sigma^\ast)$, where $\mathrm{succ}_p (\sigma^\ast)$ denotes the set of immediate successors of $\sigma^\ast$ in $p$. The operation $*$ is easily extended to $\omega^\omega$ by $a^* = \bigcup_{n \in \omega} (a \restriction n)^*$, for $a \in \omega^\omega$.

The set of branches through $p$ is $[p] = \{x \in \omega^\omega\ |\ \forall n\ (x \restriction n) \in p\}$.

We also define, for a Laver tree $p \subseteq \omega^{<\omega}$, and $\sigma \in \omega^{<\omega}$, 
\[p \ast \sigma = \left\{s \in p\ |\ \sigma^* \subseteq s \text{ or } s \subseteq \sigma^*\right\}.\]

The Laver forcing satisfies \emph{Axiom A} \cite[Definition 7.1.1]{bartoszynski1995set}. Axiom A forcing notions are pretty well-behaved in the sense that, it is possible to do a \emph{fusion argument} for them, which can often easily be extended to their iterations.

Let $L_n (p) = \{\sigma^\ast \ |\ \sigma \in n^n\}$ denote the \textit{$n$-th diagonal level of $p$}, for each $n \in \omega$. 

Now a sequence $(\leq_n)_{n \in \omega}$ of partial orders is defined as follows: for every $n \in \omega$,
\[q \leq_n p \Leftrightarrow q \leq p \text{ and } L_n (q) = L_n (p).\]

This way, if $(p_n)_{n \in \omega}$ is a sequence of Laver trees such that $p_{n+1} \leq_n p_n$, for all $n \in \omega$, then $q = \bigcap_{n \in \omega} p_n$ is a Laver tree. 

Finally, recall the \emph{pure decision property} for Laver forcing: if $p \in \mathbb{L}$ and $\varphi$ is a formula of the forcing language, then there exists a stem-preserving extension $q \leq p$ such that $q$ decides $\varphi$.

Let $\dot{x}$ be a name for an element of $2^\omega$ and $p$ be a condition forcing it. Roughly speaking, the backbone of $\dot{x}$ is composed of a sequence of ground model reals (the guiding reals) that approximate $\dot{x}$ in a helpful manner.

\begin{claim}\label{guidingreal} There exists a stem-preserving extension $q \leq p$ with the following property: for every $\sigma \in \omega^{<\omega}$, there exists a ground model real $x_\sigma \in 2^\omega$ such that, for all $k \in \omega$,
\[q \ast \sigma^\smallfrown k \Vdash \dot{x} \restriction (|\sigma| + k) = x_\sigma \restriction (|\sigma| + k).\]
\end{claim}

\begin{proof}

Note that for every $r \leq p$ and $\sigma \in \omega^{<\omega}$, one may find a stem-preserving extension $r_k \leq r \ast \sigma^\smallfrown k$ that decides $\dot{x} \restriction (|\sigma|+ k)$, for each $k \in \omega$. Fix $(x_k)_{k \in \omega}$ a sequence such that $x_k \in [\dot{x}_{r_{k}}]$, for every $k \in \omega$, where $\dot{x}_{r_{k}}$ is the longest initial segment decided by $r_{k}$. Since the space $2^\omega$ is compact, there exists $I \in [\omega]^{\omega}$ such that $(x_k)_{k \in I}$ converges and we let $x_\sigma = \lim_{k \in I} x_k$, which is defined for the condition $r' = \bigcup_{k \in I} r_k$.

From this observation, one may construct a fusion sequence $(p_n)_{n \in \omega}$ such that, for each $n, k \in \omega$, and $\sigma \in n^n$, there exists $x_\sigma \in 2^\omega$ such that
\[p_n \ast \sigma^\smallfrown k \Vdash \dot{x} \restriction (|\sigma| + k) = x_\sigma \restriction (|\sigma| + k).\]
Then $q = \bigcap_{n \in \omega} p_n$ is our desired condition. \qedhere

\end{proof}

The real $x_\sigma$ is called the \textit{$\sigma$-guiding real}. This automatically gives us a continuous ground model function $f: [q] \rightarrow 2^\omega$, defined by
\[f(a^*) = \lim x_{a \restriction n},\]
for each $a \in \omega^\omega$ (hence $a^* \in [q]$), such that $q \Vdash f(x_{\mathrm{gen}}) = \dot{x}$. This shows that Laver forcing has the \textit{continuous reading of names}  \cite[Definition 3.1.1.]{zapletal2008forcing}. 

For this reason, let us assume that $p$ is already chosen so that $x_\sigma$ is defined for $p$, for every $\sigma \in \omega^{<\omega}$. It turns out that $\dot{x}$ is ground model iff $\dot{x} = x_\sigma$, for some $\sigma \in \omega^{<\omega}$. If this is not the case, then it is possible to define a $p$-rank on $\omega^{<\omega}$ as follows: for $\sigma \in \omega^{<\omega}$,
    \begin{align*}
        \begin{cases}
        r (\sigma) = 0 &\leftrightarrow \exists^\infty n \in \omega\ \left(x_\sigma \neq x_{\sigma^\smallfrown n}\right)\\
        r (\sigma) = k > 0 &\leftrightarrow \neg\ r (\sigma) < k \land \exists^\infty n \in \omega\ \left( r (\sigma^\smallfrown n) < k\right).
        \end{cases}
    \end{align*}

There may be no levels of $\omega^{<\omega}$ for which every node at this level has $p$-rank zero. However, there will be \textit{frontiers} with this property:

Say that an antichain $A \subseteq p$ is a \textit{frontier of $p$} iff every branch through $p$ has exactly one initial segment in $A$ --- i.e., for every $x \in [p]$, there exists a unique $n \in \omega$ such that $x \restriction n \in A$. Say that $A \subseteq \omega^{<\omega}$ is a $p$-frontier iff $A^* = \{\sigma^*\ |\ \sigma \in \omega^{<\omega}\}$ is a frontier of $p$. A sequence $(A_n)_{n \in \omega}$ of $p$-frontiers forms a $p$-\textit{chain} iff for all $\sigma \in A_{n+1}$, there exists a unique $\tau \in A_n$ such that $\tau \subsetneq \sigma$.

\begin{claim}(\cite[Theorem 16]{BKW}) There exists a $p$-chain $(A_n)_{n \in \omega}$, each consisting of rank zero nodes, such that $x_{\tau \restriction m} = x_\sigma$, for all $\sigma \subsetneq \tau$ with $\sigma \in \mathrm{succ}(A_n)$ and $\tau \in \mathrm{succ}(A_{n+1})$; and all $m \in \omega$ with $|\sigma| \leq m < |\tau|$.
\end{claim}

Without loss of generality, assume $p$ is the Laver tree generated by the initial segments of these frontiers. From this and a pruning argument (for a detailed proof see \cite[Theorem 7]{groszek_1987}), we already have that $f \restriction [p]$ is injective, which implies the so-called \textit{minimality} for the Laver forcing. Moreover, by identifying the nodes of $p$ having the same guiding real, assume that every $\sigma \in \omega^{<\omega}$ has $p$-rank zero. This can be done as follows:

Let $\sim$ be the equivalence relation on $\omega^{<\omega}$ defined by 
\[\sigma \sim \tau \leftrightarrow x_\sigma = x_\tau,\]
for each $\sigma, \tau \in \omega^{<\omega}$; let  $\faktor{p}{\sim}$ be the set of equivalence classes, and $\pi: \omega^{< \omega} \rightarrow \faktor{p}{\sim}$ be the projection $\sigma \mapsto [\sigma]_{\sim}$, for each $\sigma \in \omega^{<\omega}$. 

Say that $\boldsymbol{a} \in \faktor{p}{\sim}$ is an immediate successor of some different class $\boldsymbol{b}$ iff there exist $\sigma \in \boldsymbol{a}, \tau \in \boldsymbol{b}$ such that $\sigma$ is an immediate successor of $\tau$. Note that, since $\dot{x}$ is not a ground-model real, then every node of $\faktor{p}{\sim}$ has infinitely many immediate successors. From this, let $j$ be a bijection identifying the nodes of $\omega^{<\omega}$ with $\faktor{p}{\sim}$.

Now, for every equivalence class $\boldsymbol{a}$, we define $\mathcal{M}(\boldsymbol{a})$ to be the set of all maximal nodes of $\boldsymbol{a}$ (i.e., $\sigma \in \mathcal{M}(\boldsymbol{a})$ iff there is no $\tau \supsetneq \sigma$ in $\boldsymbol{a}$). Let $G_{\boldsymbol{a}}: \omega \rightarrow \mathcal{M}(\boldsymbol{a})$ be an enumeration of $\mathcal{M}(\boldsymbol{a})$ such that $G_{\boldsymbol{a}}^{-1} (\{\sigma\})$ is infinite, for every $\sigma \in \mathcal{M}(\boldsymbol{a})$ (i.e., it enumerates $\mathcal{M}(\boldsymbol{a})$ with infinitely many repetitions). 

Now, note that if $i : \omega^{< \omega} \rightarrow \omega^{<\omega}$ is such that $i(\emptyset) = \emptyset$; and $j(i(\sigma^{\smallfrown}n)) = [\tau]_\sim$, for some $\tau$ immediate successor of $G_{[j(i(\sigma))]}(n)$, then $(\pi^{-1} [\mathrm{ran} (j \circ i)])^{*}$ is a stem-preserving Laver subtree of $p$. In our proofs, a function $i$ with this property is constructed and, with this, we know how to pull-back from the equivalence classes to a Laver subtree of $p$. For this reason, from now on we shall always assume that $p$ is defined to have rank $0$ on every node extending the stem, and convey that it is always possible to run the argument above using frontiers along with infinite repetition enumeration of their nodes.

The \textit{Laver model} is the generic extension obtained by forcing with a countable support iteration of $\mathbb{L}$, of length $\omega_2$, over the universe $V$.

\section{Cooking-up names for Borel independent sets}\label{mainsection}

If $c: X \rightarrow \alpha$ is a $G$-coloring of $X$, the \emph{maximal $G$-monocromatic} sets are the sets of the form $c^{-1}(\{\beta\})$, for $\beta < \alpha$, which satisfy the property $c^{-1}(\{\beta\})^2 \cap G = \emptyset$. The sets $A \subseteq X$ satisfying this property (i.e., $A^2 \cap G = \emptyset$) are said to be \textit{$G$-independent}, and using this we shall redefine Borel chromatic numbers in a more convenient way:  

The \textit{Borel chromatic number} of $G$, $\chi_B (G)$, is the least cardinality of a family $\mathcal{F}$, consisting of Borel $G$-independent sets, such that $\bigcup \mathcal{F} = X$. 

Recall that a graph $G$ on a Polish space $X$ is an $F_\sigma$ graph iff there exists a sequence $(C_n)_{n \in \omega}$, of closed subsets of $X^2 \setminus \mathrm{Id}_X$ such that $G = \bigcup_{n \in \omega} C_n$, where $\mathrm{Id}_X$ is the identity on $X$. It is locally countable iff the set $\{y \in X\ |\ (x, y) \in G\}$ is countable, for every $x \in X$.

If $(C_n)_{n \in \omega}$ is a cover of $G$, the \emph{$G$-locator of $(C_n)_{n \in \omega}$}, $\ell: X^2 \rightarrow \omega$, is defined by
    \begin{equation*}
        \ell(x, y) = 
        \begin{cases}
            \mathrm{min}\{n+1\ |\ (x, y) \in C_n\}, &\text{if $(x, y) \in G$}\\
            0, &\text{if $x = y$}.\\
            \infty, &\text{if $(x, y) \notin G \cup \mathrm{Id}_X$}.
        \end{cases}
    \end{equation*}
Say that $G$ is \emph{$\ell$-unbounded} iff for every $(x, y) \in G$, and $n$ such that $\ell (x, y) = n$, there exists an open set $O$ around $x$ such that $\ell(z, y) > n$, for every $z \in O \setminus \{x\}$.

We also let $\ell (A, B) = \min \{\ell(a, b)\ |\ (a, b) \in A \times B\}$, for $A, B \subseteq X$. Note that $\ell$ is identically $\infty$ only on the $G$-independent sets.

\begin{thm}\label{laverthm} Let $G$ be an $F_\sigma$ graph, with closed cover $(C_n)_{n \in \omega}$, defined on a totally disconnected compact Polish space $X$. 
\begin{itemize}
    \item[(a)] If $G$ is locally countable then, in the Miller model, every point in the completion of $X$ is contained in a Borel $G$-independent set coded in the ground model; and
    
    \item[(b)] If $G$ is $\ell$-unbounded then, in the Laver model, every point in the completion of $X$ is contained in a Borel $G$-independent set coded in the ground model.
\end{itemize}

Hence $\chi_B (G) \leq |2^{\aleph_0} \cap V|$ in each of the above cases.
\end{thm}

This way, starting from $V$ a model of CH, we obtain:

\begin{cor}\label{chiE0b}
It is consistent with ZFC that $\chi_B (E_0) < \mathfrak{d}$; and that $\chi_B (E_0) < \mathfrak{b}$.
\end{cor}

It is well-known that Laver forcing increases $\mathfrak{b}$, the \textit{bounding number} and that the Miller forcing increases $\mathfrak{d}$, the \textit{dominating number}. For a diagram involving common small cardinal characteristics of the continuum, and a few Borel chromatic numbers, see \cite{gasgesch2022}.

The reason why Theorem \ref{laverthm} can be proved for totally disconnected compact Polish spaces is because they are the continuous injective image of $2^\omega$ when they lack isolated points: i.e., if $X$ is a compact totally disconnected Polish space without isoalted points, then there exists a continuous injection $f: 2^\omega \rightarrow X$ such that $f[2^\omega] = X$. From this, if $G$ is an $F_\sigma$ locally countable graph on $X$, then \[f^*[G] = \left\{\left(f^{-1} (x), f^{-1} (y)\right) \in (2^\omega)^2\ |\ (x, y) \in G\right\}\] is an $F_\sigma$ locally countable graph on $2^\omega$; and $\chi_B (f^*[G]) = \chi_B (G)$. 

In order to prove Theorem \ref{laverthm}, we first investigate what happens when we add only one generic real to the universe. It turns out that for every real number of the generic extension, there exists a continuous injective function, coded in the ground model, whose image is a maximal monochromatic set and contains this number (Lemma \ref{laverlem}). 

As said before, it will be sufficient to consider $X = 2^\omega$, and then we may pull-back from $X$ to $2^\omega$ using a continuous homeomorphism, when $X$ is perfect:

\begin{lem}\label{laverlem} Let $G$ be an $F_\sigma$ graph on $2^\omega$, with closed cover $(C_n)_{n \in \omega}$. Then there exists a stem-preserving extension $q \leq p$ such that $f[q]$ is $G$-independent
\begin{itemize}
    \item[(a)] for the Miller forcing, if $G$ is locally countable; and
    \item[(b)] for the Laver forcing, if $G$ is $\ell$-unbounded.
\end{itemize}
\end{lem}

\begin{proof} 
    We first define an order-preserving injection $i: \omega^{<\omega} \rightarrow \omega^{<\omega}$, and a strictly increasing sequence $(k_{n})_{n \in \omega}$ of natural numbers such that for all $\sigma, \tau \in n^{\leq n}$,
    \begin{itemize}
        \item[(1)] $\ell\left(\left[x_{i(\sigma)} \restriction |i(\sigma)| + k_{n}\right], \left[x_{i(\tau)}\restriction |i(\tau)| + k_{n}\right]\right) \geq |\sigma| - |\tau|$, if $\tau \subseteq \sigma$; and
        
        \item[(2)] $\ell\left(\left[x_{i(\sigma)} \restriction |i(\sigma)| + k_{n}\right], \left[x_{i(\tau)} \restriction |i(\tau)| + k_{n}\right]\right) \geq |\sigma|+|\tau| - 2|\Delta(\sigma, \tau)|$, if $\sigma$ and $\tau$ are distinct, where $\Delta(\sigma, \tau)$ denotes the longest common initial segment of $\sigma$ and $\tau$.

    \end{itemize}
    Once this is done, we get that $q = \mathrm{ran}(i)^*= \{i(\sigma)^*\ |\ \sigma \in \omega^{<\omega}\}$ is our desired condition (i.e., a Miller or a Laver tree, depending which case we are considering).
    
    From this it easily follows that, if $a, b \in [q]$ are distinct, then $f(a)$ and $f(b)$ do not form an edge: in fact, for every $n \in \omega$, there exists $\sigma_{a, n}, \sigma_{b, n}$ such that $|\sigma_{a, n}| = |\sigma_{b, n}| = n+1$, $i(\sigma_{a, n})^* \subseteq a$ and $i(\sigma_{b, n})^* \subseteq b$. Then
    \[\ell(f(a), f(b)) \geq \ell\left(x_{i(\sigma_{a, n})}, x_{i(\sigma_{b, n})}\right) \geq 2(n + 1 - |\Delta(\sigma_{a,n}, \sigma_{b, n})|);\]
    and the sequence $|\Delta(\sigma_{a, n}, \sigma_{b, n})|$ is constant. Hence, $\ell (f(a), f(b)) = \infty$.
    
    Assume $i \restriction n^{\leq n}$ has been defined and let $\prec$ denote the lexicographic order on $\omega^{<\omega}$. By induction on $\sigma$, also assume $i (\tau)$ has been defined, for all $\tau \prec \sigma$.
    
    \begin{itemize}
        \item[(a)] From local countability, there exists a branch $a$ with $i(\sigma \restriction |\sigma|-1) \subseteq a$, and hence a long enough initial segment $i(\sigma) \subseteq a$ such that
            \begin{align*} 
            \ell\left(x_{i(\sigma)}, x_{i(\tau)}\right) >
                \begin{cases}
                    |\sigma| - |\tau|, &\text{ if $\tau \subseteq \sigma$; and}\\
                    |\sigma| + |\tau| - 2|\Delta(\sigma, \tau)|, &\text{ if $\tau$ and $\sigma$ are incompatible}.
                \end{cases}
            \end{align*}
        \item[(b)] From $\ell$-unboundedness, one can get the same as the item above, but $i(\sigma)$ can be chosen as an immediate successor of $i(\sigma \restriction |\sigma|-1)$.
    \end{itemize}

    In any case, it follows from the closedness of the $C_n$'s that there exists a natural number $k_{n+1}$ such that
    \[\ell\left(\left[x_{i(\sigma)}\restriction |i(\sigma)| + k_{n+1}\right],\left[x_{i(\tau)}\restriction |i(\tau)|+k_{n+1}\right]\right) \geq \ell\left(x_{i(\sigma)}, x_{i(\tau)}\right). \qedhere\]
\end{proof}

Our goal now is to prove some version of Lemma \ref{laverlem} for countable support iterations of the Laver forcing. The proof of this lemma will be useful to justify the proof of Claim \ref{faithfulness2}.

For an ordinal $\alpha \geq 1$, let $\mathbb{L}_{\alpha}$ denote the countable support iteration of $\mathbb{L}$. Let $F$ be a finite subset of $\alpha$ and $\eta: F \rightarrow \omega$. Say that $q \leq_{F, \eta} p$ iff  
\[\forall \gamma \in F \left(q \restriction \gamma \Vdash q(\gamma) \leq_{\eta(\gamma)} p(\gamma)\right).\]

Now if $\dot{x}$ is a name for an element of $2^\omega$ not added at proper stage of the iteration and $p$ is a condition forcing it, then we may define an iterated version of the guiding reals:


\begin{claim} For every $\gamma < \alpha$ and $\sigma \in \omega^{<\omega}$, there exists an $\mathbb{L}_\alpha$-condition $p_\sigma^\gamma \leq p$, and an $\mathbb{L}_\gamma$-name for a real $x_\sigma^\gamma$, such that $p_\sigma^\gamma \restriction \gamma$ forces that $p^{\gamma}_{\sigma}(\gamma) \leq_{0} p(\gamma)$; and
    \[p_\sigma^\gamma (\gamma) \ast (\sigma^\smallfrown k)^\smallfrown p_\sigma^\gamma \restriction (\gamma, \alpha) \Vdash \dot{x} \restriction (|\sigma| + k) = x_\sigma^\gamma \restriction (|\sigma| + k),\]
for all $k \in \omega$.
\end{claim}

\begin{proof}
    Note that if $\varphi$ is a formula of the forcing language and $q \leq p$, then by the virtue of the pure decision property there exists $r \leq q$ such that
    \[r \restriction \gamma \Vdash r(\gamma) \leq_0 q(\gamma) \ \text{and $r \restriction [\gamma, \alpha)$ decides $\varphi$}.\]
    This way we get $p_\sigma^\gamma \leq p$ such that, for every $k \in \omega$: 
    \[p_\sigma^\gamma \restriction \gamma \Vdash p^\gamma_\sigma (\gamma) \ast (\sigma^\smallfrown k)^\smallfrown p_\sigma^\gamma \restriction (\gamma, \alpha) \text{ decides }\dot{x} \restriction (|\sigma| + k).\]

    The definition of the $\mathbb{L}_\gamma$ now follows from compactness of $2^\omega$, as in the proof of the single case (Claim \ref{guidingreal}).
\end{proof}

Moreover, one may construct a sequence $(p_\sigma^\gamma)_{\sigma \in \omega^{<\omega}}$ such that each $p_\sigma^\gamma$ is as above, and $p_{\sigma^\smallfrown n}^\gamma \leq p_\sigma^\gamma$ for all $n \in \omega$ and $\sigma \in \omega^{<\omega}$. From this it is possible to define an \textit{iterated rank at coordinate $\gamma$}: for $\sigma \in \omega^{<\omega}$,

\[r^\gamma (\sigma) = 0 \leftrightarrow \exists^\infty k \in \omega\ \left(p_{\sigma^\smallfrown k}^\gamma \restriction \gamma \Vdash x_\sigma^\gamma \neq x_{\sigma^\smallfrown k}^\gamma\right).\]

Positive ranks are defined similarly to the single case. 

Let $F \subseteq \alpha$, $\eta: F \rightarrow \omega$, and $\sigma \in \prod_{\gamma \in F} \eta(\gamma)^{\eta(\gamma)}$. We define $p \ast \sigma$ such that
\[\forall \gamma \in F \left((p \ast \sigma) \restriction \gamma \Vdash (p \ast \sigma) (\gamma) = p(\gamma) \ast \sigma(\gamma) \right).\]

Say that $q \leq p$ is \textit{$G$-$(F, \eta)$-faithful} iff 
\[\ell\left([\dot{x}_{q \ast \sigma}], [\dot{x}_{q \ast \tau}]\right) \geq \ell_{\max} \doteq \max_{\gamma \in F} \{|\sigma(\gamma)| + |\tau(\gamma)| - 2|\Delta(\sigma(\gamma), \tau(\gamma))|\},\] 
for all distinct $\sigma, \tau \in  \prod_{\gamma \in F} \eta(\gamma)^{\eta(\gamma)}$, where $\dot{x}_{r}$ is the maximal initial segment decided by $r \leq q$.

Let $\eta'$ be defined such that  $\eta'(\gamma) = \eta(\gamma)$, for all $\gamma \notin \{\beta,\bar{\gamma}\}$;  $\eta' (\beta) = \eta(\beta) + 1$ if $\bar{\gamma} \neq \beta$; and $\eta'(\bar{\gamma}) = \eta_{\max} + \ell_{\max}+1$, where $\eta_{\max} = \underset{\gamma \in F}{\max} \ \eta(\gamma)$.\\  

\begin{lem}\label{faithfulness2} Let $G$ be an $F_\sigma$ graph on $2^\omega$, with closed cover $(C_n)_{n \in \omega}$, and $q \leq_{F, \eta} p$ be a $G$-$(F, \eta)$-faithful condition. Then there exists a $G$-$(F, \eta')$-faithful condition $r \leq_{F, \eta} q$,
\begin{itemize}
    \item[(a)] for countable support iterations of the Miller forcing; if $G$ is locally countable; and
    \item[(b)] for countable support iterations of Laver forcing, if $G$ is $\ell$-unbounded.
\end{itemize}
\end{lem}

\begin{proof}
    Let $\{\sigma_0, ..., \sigma_{m-1}\}$ be an enumeration of $\prod_{\gamma \in F\setminus\{\bar{\gamma}\}} \eta(\gamma)^{\eta (\gamma)}$. We define a $\leq_{F, \eta}$-decreasing sequence $(p_j)_{j < m}$, as follows: 

    Assume we have constructed $p_{j-1}$. Since $\dot{x}$ is not added at any proper stage of the iteration, there exists $q_j \leq_{F, \eta} p_{j-1}$ such that all $\tau \in \omega^{\leq \eta_{\max} + \ell_{\max} + 1}$ have $\bar{\gamma}$-rank zero for the condition $q_j \ast \sigma_j$. 

    Using ideas from the proof of Lemma \ref{laverlem}, we define an order-preserving injection $i$ on all the set of all $\tau$'s as above, and find a condition $p_j \leq_{F, \eta} q_j$ such that $(p_j \ast \sigma_j) \restriction \bar{\gamma}$ forces that

        \begin{itemize}
            \item[(1)] $\ell\left(\left[x_{i(\tau)}^{\bar{\gamma}} \restriction |i(\tau)| + k_{n}\right], \left[x_{i(\tau')}^{\bar{\gamma}} \restriction |i(\tau')|+ k_n\right]\right) \geq |\tau| - |\tau'|$, if $\tau' \subseteq \tau$; and
                
            \item[(2)] $\ell\left(\left[x_{i(\tau)}^{\bar{\gamma}}\restriction |i(\tau)|+ k_n \right], \left[x_{i(\tau')}^{\bar{\gamma}} \restriction |i(\tau')|+ k_n \right]\right) \geq |\tau|+|\tau'| - 2 |\Delta(\tau, \tau')|$, if $\tau$ and $\tau'$ are incompatible.
        \end{itemize}
        
    for all $\tau, \tau' \in \mathrm{dom} (i)$. In particular, \[\ell\left(\left[x_{i(\tau)}^{\bar{\gamma}} \restriction |i(\tau)| + k_{\bar{n}}\right], \left[x_{i(\tau')}^{\bar{\gamma}} \restriction |i(\tau')|+ k_{\bar{n}}\right]\right) \geq \ell_{\max} + 1\] 
    when $|\tau| = |\tau'| = \lceil\eta_{\max} + (\ell_{\max} +1)/2 \rceil$, and $|\Delta(\tau, \tau')| \leq \eta_{\max}$.
    
    If $\beta = \bar{\gamma}$, simply let $r = p_{m-1}$; if $\beta \neq \bar{\gamma}$, let $\left\{I_{\tau}\ |\ \tau \in \eta'(\beta)^{ \eta'(\beta)}\right\}$ denote a partition of $\omega$ into finitely many infinite pieces. Then $r \leq_{F, \eta} p_{m-1}$ is defined such that 

\begin{itemize}
    \item[(1)] $r \restriction \bar{\gamma} = p_{m-1} \restriction \bar{\gamma}$; 
    
    \item[(2)] for all coordinatewise extensions $\sigma' \in \prod_{\gamma \in F \setminus \{\bar{\gamma}\}} \eta'(\gamma)^{\eta'(\gamma)}$, of the restricted product of nodes $\sigma \in \prod_{\gamma \in F \setminus \{\bar{\gamma}\}} \eta(\gamma)^{\eta(\gamma)}$, for all $\bar{\sigma} \in \eta (\bar{\gamma})^{< \eta(\bar{\gamma})}$,
    \[(r \ast \sigma') \restriction \bar{\gamma} \Vdash \mathrm{succ}(\mathrm{st} (r(\bar{\gamma})\ast \bar{\sigma}) \setminus \{0, ..., \eta(\bar{\gamma}) - 1\}^* = I_{\sigma'(\beta)}^{\ast},\]
    where $\{0, ..., k-1\}^*$ denotes the first $k$ immediate successors of the stem of the restriction of $r(\bar{\gamma})$ to $\bar{\sigma}$, $r(\bar{\gamma})\ast \bar{\sigma}$; for all $\bar{\sigma} \in \eta (\bar{\gamma})^{\eta(\bar{\gamma})}$,
    \[(r \ast \sigma') \restriction \bar{\gamma} \Vdash \mathrm{succ}(\mathrm{st} (r(\bar{\gamma})\ast \bar{\sigma}) = I_{\sigma'(\beta)}^{\ast};\]

    \item[(3)] and $r \restriction (\bar{\gamma}+1) \Vdash r \restriction (\bar{\gamma}, \alpha) = p_{m-1} \restriction (\bar{\gamma}, \alpha)$ \qedhere.
\end{itemize}

\end{proof}

\begin{proof}[Proof of Theorem \ref{laverthm}]
    Using Lemma \ref{faithfulness2} and some bookkeeping, we may construct a sequence $(p_n, F_n, \eta_n)_{n \in \omega}$ such that
    \begin{itemize}
        \item[(1)] $F_n \subseteq \alpha$ is finite; 
        \item[(2)] $F_n \subseteq F_{n+1}$;
        \item[(3)] $\eta_n: F_n \rightarrow \omega$;
        \item[(4)] $\eta_n (\gamma) \leq \eta_{n+1} (\gamma)$, for all $\gamma \in F$;
        \item[(5)] $p_{n+1} \leq_{F_n, \eta_n} p_n$;
        \item[(6)] for all $\gamma \in \mathrm{supp}(p_n)$, there is $m \in \omega$ such that $\gamma \in F_m$ and $\eta_m (\gamma) \geq n$; and
        \item[(7)] $p_n$ is $(F_n, \eta_n)$-faithful.
    \end{itemize}
    
    Let $q \in \mathbb{L}_\alpha$ be defined recursively such that
    \[\forall \gamma < \alpha \left(q \restriction \gamma \Vdash q(\gamma) = \bigcap_{n \in \omega} p_n (\gamma)\right).\]

    Let $(x(\gamma))_{\gamma \in \mathrm{supp}(q)}$ be a sequence in $(\omega^{\omega})^{\mathrm{supp}(q)}$ and define a function $f$ by 
    \[f \left(\left(x(\gamma)_{\gamma \in \mathrm{supp}(q)}\right)\right) = \bigcup_{n \in \omega} \dot{x}_{q\ast(x(\gamma)\restriction\eta_{n}(\gamma))_{\gamma \in F_{n}}}.\] 
    
    This is a ground model continuous injection $f: (\omega^\omega)^{\mathrm{supp}(q)} \rightarrow 2^\omega$ 
    mapping the generic sequence to $\dot{x}$ --- i.e., $q \Vdash f(x_\mathrm{gen} (\gamma))_{\gamma \in \mathrm{supp}(q)} = \dot{x}$. Due to the above property, we have $\ell (f(x), f(y)) = \infty$, for all distinct $x, y \in (\omega^\omega)^{\mathrm{supp} (q)}$. Hence, $f\left[(\omega^\omega)^{\mathrm{supp}(q)}\right]$ is a ground model Borel $G$-independent set. \qedhere
\end{proof}

\section{An application to regularity properties}\label{regularities}

Regularity properties emerged as early as the discipline of descriptive set theory. The two main examples of this type of good behavior of sets of reals are the notions of Lebesgue and Baire measurabilities. 

From the close relationship between these measurability notions and the random and the Cohen forcings, various new notions of measurability emerged. Namely, if $\mathbb{P}$ is a forcing notion of perfect subtrees, either of $2^{<\omega}$ or $\omega^{<\omega}$, and $A$ is a subset of the respective space in which the forcing is defined, we say that $A$ is \textit{$\mathbb{P}$-measurable} iff
\[\forall p \in \mathbb{P}\ \exists q \leq p\ \left([q] \subseteq A \text{ or } [q] \cap A = \emptyset\right).\]
From this it is possible to isolate the notion of \textit{$\mathbb{P}$-null} sets:
\[p^0 = \{A\ |\ \forall p \in \mathbb{P}\ \exists q \leq p\ ([q] \cap A = \emptyset)\},\]
where $[p]$ denotes the set of branches through $p$ (i.e., the set of all $x \in \omega^{\omega}$ such that for all $n \in \omega$, $x \restriction n \in p$).

Using typical fusion arguments, it is often possible to prove that $p^0$ is a $\sigma$-ideal, which is the case for the three forcing notions considered here.

If $\mathbb{P}$ is the random forcing, $\mathbb{P}$-measurability is equivalent to the Lebesgue measurability, and if $\mathbb{P}$ is the Cohen forcing, then $\mathbb{P}$-measurability is equivalent to the Baire measurability. To answer the question of Fischer, Friedman and Khomskii we need the notions of $\mathbb{E}_0$ and Silver measurabilities:

A tree $p \subseteq 2^{<\omega}$ is an \textit{$E_0$-tree} iff it is perfect; and for every splitting node $s \in p$, there are $s_0 \supseteq s^\smallfrown 0$ and $s_1 \supseteq s^\smallfrown 1$, of the same length, such that
\[\left\{x \in 2^\omega\ |\ s_0^\smallfrown x \in [p]\right\} = \left\{x \in 2^\omega\ |\ s_1^\smallfrown x \in [p]\right\}.\]
If $s_0$ and $s_1$ can always be chosen to be $s^\smallfrown 0$, and $s^\smallfrown 1$, respectively, then $p$ is a \textit{Silver tree}. The \textit{$E_0$-forcing}, $\mathbb{E}_0$, consists of $E_0$-trees; and the \textit{Silver forcing}, $\mathbb{V}$, consists of Silver trees.

There are various implications between regularity properties \cite[§4 pg. 1350]{brendle2012polarized}, as well as consistent separations. Since Silver measurability clearly implies $E_0$-measurability, we show how to separate Laver and $E_0$-measurabilities with the help of Ikegami's theorem. For that, we need the definition of \emph{quasi-generic real}: say that a real number $x$ is \textit{$\mathbb{P}$-quasi-generic over $M$}, a model of ZF, iff $x \notin B$, for all Borel set $B \in p^0 \cap M$ (i.e., Borel null sets \textit{coded} in $M$).

It is always the case that generic reals are quasi-generic reals, but not all quasi-generic reals are generic reals. Furthermore, we can immediately see, from Theorem \ref{laverthm}, that Laver forcing does not add $E_0$-quasi-generic reals (hence, it also does not add Silver quasi-generic reals). 

\begin{fact}[Theorem 1.3 of \cite{ike1}]\label{ik} For $\mathbb{P}$ satisfying a stronger form of properness and additional definability requirements \cite[Definitions 2.3 and 2.4]{ike1}: every $\boldsymbol{\Delta}^{1}_{2}$ set of reals is $\mathbb{P}$-measurable if, and only if, there exists a $\mathbb{P}$-quasi-generic real over $L[x]$, for each $x \in 2^\omega$.
\end{fact}

From this, together with the the fact that $\boldsymbol{\Delta}^1_2 (\mathbb{L})$ implies $\boldsymbol{\Sigma}^1_2 (\mathbb{L})$, we are able to answer Question 6.3. of \cite{Fischer2014CichosDR}:

\begin{cor}\label{silvervslaver} $\boldsymbol{\Sigma}^{1}_{2}(\mathbb{L}) \land \neg \boldsymbol{\Delta}^{1}_{2}(\mathbb{E}_0)$ holds in the model obtained by forcing with an $\omega_1$-iteration of $\mathbb{L}$, with countable support, over $L$.
\end{cor}

It is well-known that Silver forcing adds splitting \cite[Proposition 2.4]{brendle_halbeisen_lwe_2005}. Brendle, Halbeisen and Löwe asked \cite[Question 2]{brendle_halbeisen_lwe_2005} whether the existence of splitting reals over $L[x]$, for every real $x$, implies $\boldsymbol{\Delta}^{1}_{2}(\mathbb{V})$. 

\begin{cor}\label{splitting} In the model from Corollary \ref{silvervslaver} $\boldsymbol{\Delta}^1_2 (\mathbb{V})$ does not hold, even though there are splitting reals over each $L[x]$.
\end{cor}

\section{Questions}\label{questions}

The first question concerns the topology of the set of vertices: we could not find a counterexample for Theorem \ref{laverthm} when the set of vertices is not compact, or not extremely disconnected. 

\begin{quest} Does Theorem \ref{laverthm} still hold if $X$ is not compact (e.g., $X = \omega^\omega$)? What if $X$ is not extremely disconnected (e.g., $X = [0, 1]$, or $X = \mathbb{R}$)?
\end{quest}

One could also ask the role of local countability. Sometimes, it suffices that the graphs lack perfect cliques, for its Borel chromatic number to be bounded by the ground model continuum, as in the case of the Sacks model for $F_\sigma$ graphs, and of the $E_0$-model, for closed graphs. 

\begin{quest} Does Theorem \ref{laverthm} still hold if we only know that $G$ lacks perfect cliques?
\end{quest}

Finally, we do know what happens for graphs of different complexities, such as $G_\delta$, $G_{\delta \sigma}$, $F_{\sigma \delta}$ etc. 

\begin{quest} Does Theorem \ref{laverthm} still hold if $G$ is an analytic graph?
\end{quest}

We think answering positively all three questions above would likely give us a simoutanous result --- i.e., Theorem \ref{laverthm} would hold for all $X$ Polish and $G$ analytic. Nevertheless, we conjecture that the compactness of $X$ is indispensable. 

On the regularity side, we have the surprising fact that $\boldsymbol{\Sigma}^1_2 (\mathbb{V})$ implies $\boldsymbol{\Sigma}^1_2 (\mathbb{M})$, where $\mathbb{M} \supsetneq \mathbb{L}$ denotes the Miller forcing \cite[Proposition 3.7]{brendle_halbeisen_lwe_2005}. It seems to be unknown whether $\boldsymbol{\Sigma}^1_2 (\mathbb{E}_0)$ already implies $\boldsymbol{\Sigma}^1_2 (\mathbb{M})$.

\begin{quest} Does $\boldsymbol{\Sigma}^1_2 (\mathbb{E}_0)$ imply $\boldsymbol{\Sigma}^1_2 (\mathbb{M})$? 
\end{quest}

On the other hand, if one wishes to prove that $\boldsymbol{\Sigma}^1_2 (\mathbb{E}_0)\land \neg \boldsymbol{\Delta}^1_2 (\mathbb{M})$ is consistent, then the common idea is to produce some \textit{amoeba} for the forcing $\mathbb{E}_0$, that will add to each $L[x]$ an $E_0$-tree of $E_0$-reals, which is $\omega^\omega$-bounding. While this may be optmitistic, $\boldsymbol{\Sigma}^1_2 (\mathbb{V})\land \neg \boldsymbol{\Delta}^1_2 (\mathbb{L})$ is true in the model obtained by adding Cohen reals over each $L[x]$. On the other hand, from the work of Spinas we know that any reasonable amoeba for Silver adds Cohen reals \cite[Theorem 5]{spinas2016silver}.

\begin{quest} Is there an amoeba for $\mathbb{E}_0$ which does not add Cohen reals?
\end{quest}

\bibliographystyle{plain}
\bibliography{bibliography}

\end{document}